\documentclass[11pt]{amsart}
\usepackage[margin=1in]{geometry}
\usepackage{xcolor}

\usepackage[colorlinks=true]{hyperref}
\usepackage{amsmath}
\usepackage{amsfonts}
\usepackage{paralist}
\usepackage{multirow}
\usepackage{array}
\numberwithin{equation}{section}
\newtheorem{thm}{Theorem}[section]
\newtheorem{prop}[thm]{Proposition}
\newtheorem{lemma}[thm]{Lemma}

\newtheorem{defn}[thm]{Definition}

\newcommand{\ot}{\Omega_T }

\newcommand{\mdiv}{\textup{div}}

\newcommand{\wot}{W^{1,2}\left(\Omega\right)}
\newcommand{\tuj}{\tilde{u}_j}
\newcommand{\buj}{\bar{u}_j}
\newcommand{\bvj}{\bar{v}_j}

\newcommand{\bz}{\beta_0}
\newcommand{\ft}{F_{\tau}}

\begin{document}
	
	\title[A Crystal Surface Model With Mobility]{A Global Existence Theorem For A Fourth-Order Crystal Surface Model With Gradient Dependent Mobility}
	\author{B. C. Price and Xiangsheng Xu}\thanks
	{Department of Mathematics and Statistics, Mississippi State
		University, Mississippi State, MS 39762.
		{\it Email}: bcp193@msstate.edu (B. C. Price); xxu@math.msstate.edu (Xiangsheng Xu).}
	 \subjclass{}
	
	\begin{abstract} 
		In this article we study the existence of solutions to a fourth-order nonlinear PDE related to crystal surface growth. The key difficulty in the equations comes from the mobility matrix, which depends on the gradient of the solution. When the mobility matrix is the identity matrix there are now many existence results, however when it is allowed to depend on the solution we lose crucial estimates in the time direction. In this work we are able to prove the global existence of weak solutions despite this lack of estimates in the time direction. 
	\end{abstract}	
	
	\maketitle

	\section{Introduction}
	
	Let $T>0$ and $\Omega$ be a bounded domain in $\mathbb{R}^2$ with a Lipshitz boundary $\partial\Omega$. We consider the problem, 
	
	\begin{align}
		\partial_t u \in -& \mdiv\left(M\left(\nabla u \right) \nabla \mdiv\left[\nabla_z E\left(\nabla u\right)\right] \right), \ \ \mbox{ in $\ot \equiv \Omega \times\left(0,T\right)$, } \label{I1} \\
		\nabla u \cdot \nu &= M\left(\nabla u \right) \nabla \mdiv\left[\nabla_z E\left(\nabla u\right)\right] \cdot \nu = 0, \ \ \mbox{ on $\Sigma_T \equiv \partial\Omega\times\left(0,T\right)$, } \label{I2} \\
		u(x,0)& = u_0(x), \ \ \mbox{ on $\Omega$, } \label{I3} 
	\end{align}
	Here $M\left(\nabla u \right)$ is a $2\times2$ coefficient matrix dependent on $\nabla u$, and $\nu$ is the unit outward normal to $\partial\Omega$. Next,
	\begin{equation}
		E\left(z\right) = \frac{1}{p} \left|z\right|^p + \bz \left|z\right|, \ \ \bz\in(0,\infty). \nonumber 
	\end{equation}
	Then, the subgradient $\nabla_z E\left(z\right)$ is given by, 
	\begin{equation}
		\nabla_z E\left(z\right) = \left\{ \begin{array}{lr} \left|z\right|^{p-2} z + \bz\left|z\right|^{-1} z, &  z\neq 0, \ z \in \mathbb{R}^N, \\
			\bz \left[-1,1\right]^N, & z=0, \end{array} \right. \nonumber 
	\end{equation}
	The subgradient of $E$ is understood in the graph sense, i.e., it is a multi-valued function, which explains the inclusion sign in equation \eqref{I1}. 
	
	Our interests in this problem come from the mathematical modeling of the evolution of a crystal surface. \cite{BCF,MW,KDM,YG,GLL,GLL2, MK}. We will first give a brief description of the derivation. 
	
	First we let $u$ be the surface height. Then the relaxation of a crystal surface below the roughing temperature can be described by the conservation law, 
	\begin{equation}
		\partial_t u + \mdiv J = 0, \label{cl}
	\end{equation}
	Where $J$ denotes the adatom flux. We let $M=M\left(\nabla u \right)$ be the mobility and $\Gamma_s$ the local equilibrium density of adatoms. Then from Fick's law \cite{MK} we find, 
	\begin{equation}
		J= - M\left(\nabla u \right) \nabla \Gamma_s. \nonumber 
	\end{equation}
	Using the Gibbs-Thomson relation we can write $\Gamma_s$ as,
	\begin{equation}
		\Gamma_s = \rho_0 e^{\frac{\mu}{kT_0}}, \nonumber 
	\end{equation}
	where $\mu$ represents the chemical potential, $\rho_0$ is the constant reference density, $k$ is the Boltzmann constant, and $T_0$ is temperature. We then consider the mobility $M\left(\nabla u \right)$ which was introduced in \cite{MK}, having the form,
	\begin{equation}
		M\left(\nabla u \right) = S \Lambda S^{T}, \nonumber 
	\end{equation}
	where,
	\begin{equation}
		S= \frac{1}{\left|\nabla u \right|}\left(\begin{array}{cc} u_x & -u_y \\	u_y & u_x \\ \end{array} \right), \ \ \Lambda = \left( \begin{array}{cc} \frac{1}{1 + q \left|\nabla u \right|} & 0 \\ 0 & 1 \\ \end{array} \right) \ \mbox{ for some $q\geq 0$. } \nonumber 
	\end{equation}
	
	By a direct calculation we then have, 
	\begin{equation}
		M\left(\nabla u \right) = \left( \begin{array}{cc} 1 + \frac{u_x^2}{\left|\nabla u \right|^2} \left(\frac{1}{1+q\left|\nabla u \right|} - 1\right) & \frac{u_x u_y}{\left|\nabla u \right|^2} \left(\frac{1}{1+q\left|\nabla u \right|} -1\right) \\ \frac{u_x u_y}{\left|\nabla u \right|^2} \left(\frac{1}{1+q\left|\nabla u \right|} -1\right) & 1 + \frac{u_y^2}{\left|\nabla u \right|^2} \left(\frac{1}{1+q\left|\nabla u \right|} - 1\right) \\ \end{array} \right)
	\end{equation}
	Crystal surfaces are able to develop facets, where $\nabla u = 0$. In order to define $M\left(\nabla u \right)$ there we first notice that for $\nabla u \neq 0$, the functions, 
	\begin{equation}
		\frac{u_x^2}{\left|\nabla u \right|^2}, \ \frac{2u_x u_y }{\left|\nabla u \right|^2 }, \ \frac{u_y^2 }{\left|\nabla u \right|^2 } \leq 1.  \nonumber 
	\end{equation}
	At the same time we also have,
	\begin{equation}
		\frac{1}{1+q\left|\nabla u \right|} - 1 = 0, \ \ \mbox{ on the set $\{\nabla = 0 \}$. } \nonumber 
	\end{equation}
	Thus in a natural way we define, 
	\begin{equation}
		M\left(\nabla u \right) = I, \ \ \mbox{ on the set $\{\nabla u = 0\}$. } \nonumber 
	\end{equation}
	Where $I$ denotes the $2\times 2$ identity matrix. Clearly then $M\left(\nabla u\right)$ is well-defined a.e.. Furthermore we also have,
	\begin{equation}
		M\left(\nabla u \right) \xi \cdot \xi \geq \frac{\left|\xi\right|^2}{1+q\left|\nabla u \right|}, \ \ \mbox{ for all $ \xi \in \mathbb{R}^2$. } \label{mc}
	\end{equation}
	We now denote by $\Omega$ our "step locations area" of interest. We then take the general surface energy $G\left(u\right)$ to be, 
	\begin{equation}
		G\left(u\right) = \frac{1}{p} \int_{\Omega} \left|\nabla u \right|^p dz + \bz \int_{\Omega} \left|\nabla u \right| dz.  \nonumber 
	\end{equation}
	The chemical potential $\mu$ is then defined to be the change per atom in the surface energy. That is to say, 
	\begin{equation}
		\mu = \frac{\delta G}{\delta u} = - \mdiv\left[\nabla_z E\left(\nabla u\right)\right]. \nonumber 
	\end{equation}
	Incorporating each of these into our conservation equation \eqref{cl} we then have the following evolution equation for $u$, 
	\begin{equation}
		\partial_t u - \mdiv\left(M\left(\nabla u \right) \nabla e^{-\mdiv\left[\nabla_z E\left(\nabla u\right)\right] } \right) = 0 \label{cl2}
	\end{equation}
	Next we linearize the exponential term, 
	\begin{equation}
		e^{-\Delta u } = 1-\mdiv\left[\nabla_z E\left(\nabla u\right)\right]. \nonumber 
	\end{equation}
	Upon using this in \eqref{cl2} we find our equation of interest,
	\begin{equation}
		\partial_t u + \mdiv\left(M\left(\nabla u \right) \nabla \mdiv\left[\nabla_z E\left(\nabla u\right)\right] \right) = 0. \label{cl3} 
	\end{equation}
	
	In the diffusion-limited regime the dynamics are primarily due to the diffusion across terraces, and one can take $M = 1$. In this case their are now many results for both the linearized and the nonlinear problem. \cite{YG, GLL, GLL2, BM, LS, LX, PX, PX2} However, if $M\neq 1$ the story is much more complicated. This is because we lose all known estimates in the time direction as soon as we include $M\left(\nabla u \right)$. There are two main results that we are aware of dealing with the case $M\neq 1$. \cite{YG, XX2} In \cite{YG} the gradient flow theory was used to analyze a similar equation, however the gradient flow theory does not seem to be applicable to \eqref{I1}. 
	
	In \cite{XX2} the related elliptic problem was studied, 
	\begin{equation}
		\mdiv\left(M\left(\nabla u \right) \nabla \mdiv\left(\left|\nabla u \right|^{p-2} \nabla u + \beta\frac{\nabla u}{\left|\nabla u \right|} \right) \right) + au = f, \ \ \mbox{ in $\Omega$. } \label{cl5} 
	\end{equation}

	While an existence theorem for the elliptic problem \eqref{cl5} was obtained in \cite{XX2} an existence theorem for the parabolic problem seemed hopeless as there were no known estimates in the time direction. This means that any sort of compactness for the exponent term in \eqref{cl2} remained out of reach. In this work we are able to prove the existence of a suitable weak solution to the parabolic problem \eqref{I1}.
	
	Our starting point for studying \eqref{I1}-\eqref{I3} is a single a priori estimate. To describe it, we first turn our fourth-order equation \eqref{I1} into a system of second-order equations,  
	\begin{align}
		\partial_t u - \mdiv\left(M\left(\nabla u \right) \nabla v \right) = 0, \ \ \mbox{ in $\ot$ }, \label{I4} \\
		-\mdiv\left[ \left|\nabla u \right|^{p-2}\nabla u + \bz \frac{\nabla u }{\left|\nabla u \right|} \right] = v, \ \ \mbox{ in $\ot$. } \label{I5} 
	\end{align}
	Then, multiply through equation \eqref{I4} by $v$ and integrate the resulting inequality over $\Omega$ to obtain, 
	\begin{equation}
		\int_{\Omega} \partial_t u v dx + \int_{\Omega} M\left(\nabla u \right) \nabla v \cdot \nabla v dx = 0. \label{I6}
	\end{equation}
	
	For the first integral in the above equation we use \eqref{I5} to calculate, 
	\begin{align}
		\int_{\Omega} \partial_t u v dx &= -\int_{\Omega} \partial_t u \  \mdiv\left[ \left|\nabla u \right|^{p-2}\nabla u + \bz \frac{\nabla u }{\left|\nabla u \right|} \right] dx \nonumber \\
		& = \frac{1}{p}\frac{d}{dt} \int_{\Omega} \left|\nabla u \right|^p dx + \bz \frac{d}{dt} \int_{\Omega} \left|\nabla u \right| dx. \label{I7}
	\end{align}
	
	Next, for the second integral in \eqref{I6} we use \eqref{mc} to find, 
	\begin{equation}
		\int_{\Omega} M\left(\nabla u \right) \nabla v \cdot \nabla v dx \geq \int_{\Omega} \frac{ \left|\nabla v \right|^2 }{1 + q \left|\nabla u \right| } dx. \label{I8}
	\end{equation}

	We then incorporate \eqref{I7} and \eqref{I8} back into \eqref{I6} and then integrate the resulting inequality over the time domain to derive the a priori estimate, 
	\begin{align}
		\sup_{0 \leq t \leq T } \int_{\Omega} \left( \left|\nabla u \right|^p + \left|\nabla u \right| \right) dx + \int_{\ot} \frac{ \left|\nabla v \right|^2 }{1 + q \left|\nabla u \right| } dxdt \nonumber \\
		\leq c \int_{\Omega} \left( \left|\nabla u_0 \right|^p + \left|\nabla u \right| \right) dx. \label{I9}
	\end{align}
	
	Unfortunately this estimate is quite weak, and one cannot hope for much from this alone. The central issue of course is that it does not provide enough information in the time direction. Because of this we run into an issue when trying to improve the weak convergence of $u$ to strong convergence in $W^{1,p}\left(\Omega\right)$. On the other hand the strong  convergence of $\nabla u$ is necessary due to the nonlinear coefficient matrix $M\left(\nabla u\right)$. The basic idea on how to upgrade the converges was first carried out in \cite{PX2}. The idea is to first find an estimate for the $\partial_t u$ in the dual space $\left(W^{1,\frac{2p}{p-1}}\left(\Omega\right)\right)^*$. This along with the a priori estimate \eqref{I9} then puts us in position to be able to use the Aubin-Lions lemma \ref{la} to obtain the strong convergence for $u$. The next piece of the puzzle comes from the fact that the p-Laplacian is a Monotone operator. We are able to exploit this to prove that $\nabla u$ converges strongly as well.

	Before we give the statement of our main theorem, we first give our definition of a weak solution to \eqref{I1}-\eqref{I2}. 
	\begin{defn}[Definition of a weak solution]\label{wsdefn}
		We say that a triple $\left(u,v,\varphi\right)$ is a weak solution to \eqref{I1}-\eqref{I3} if the following conditions hold;
		
		(D1) $u \in L^{\infty}\left(0,T; W^{1,p}\left(\Omega\right)\right)$, and $\partial_t u \in L^{2} \left(0,T; \left(W^{1,\frac{2p}{p-1}}\left(\Omega\right) \right)^* \right)$;
		
		(D2) $v \in L^{2}\left(0,T;W^{1,\frac{2p}{p+1}}\left(\Omega\right)\right)$;
		
		(D3) $\varphi \in \left(L^{\infty}\left(\ot\right)\right)^2$, and $\varphi \in \partial_z H\left(\nabla u\right)$, where $H(z) = \left|z\right|$. 
		
		(D4) $\left(u, v, \varphi\right)$ satisfy the integral equations, 
		\begin{align}
			-\int_{\ot} &u \partial_t \psi_1 dxdt + \langle u(x,T), \psi_1(x,T) \rangle  + \int_{\ot} M\left(\nabla u \right) \nabla v \cdot \nabla \psi_1 dxdt = \int_{\Omega} u_0(x) \psi_1\left(x,0\right) dx, \label{d1} \\
			&\int_{\ot} \left|\nabla u \right|^{p-1} \nabla u \cdot \nabla \psi_2 dxdt + \bz \int_{\ot} \varphi \cdot \psi_2 dxdt = \int_{\ot} v \psi_2 dxdt \label{d2}
		\end{align}
		For each smooth pair $\left(\psi_1,\psi_2\right)$. 
	\end{defn}

	Before we proceed we would like to make a few remarks about definition \ref{wsdefn}. In order to give a precise meaning to the one-Laplacian, we have followed the tradition in monotone operator theory, that is, we say that, 
	\begin{equation}
		\varphi = \frac{\nabla u}{\left|\nabla u \right|} \iff \varphi(x,t) \in \partial_z H\left(\nabla u(x,t)\right), \ \ \mbox{ for a.e. $(x,t) \in \Omega_T$. }
	\end{equation}
	See also \cite{XX,XX2} where this approach was taken. 
	
	With this definition in hand we then have following theorem, 
	\begin{thm}[Main Theorem]\label{maintheorem}
		Suppose that $T>0$ and $\Omega$ is a bounded domain in $\mathbb{R}^2$ with a Lipshitz boundary $\partial\Omega$. We assume also that $u_0(x)$ is a given function in the space $W^{1,2}\left(\Omega\right)$. Then there is a global weak solution to \eqref{I1}-\eqref{I3} on $\ot$ in the sense of definition \ref{wsdefn}. 
	\end{thm}
	
	The solution in \ref{maintheorem} will be constructed as the limit of a sequence of suitable approximate solutions. To basis of our approximate problems is first discretizing \eqref{I1} in the time direction, and then adding in lower order terms for a regularizing affect. We also add an additional term to avoid any degeneracy in $M\left(\nabla u \right)$.

	The rest of this work is organized as follows. In the next section we lay out the basic notation and tools that we will use throughout the rest of the paper. In the third section we present and prove the existence of solutions to our approximate problems. The existence assertion for our approximate problems comes from the Leray-Schauder Theorem \ref{LSThm}. Then in the fourth and final section we give the proof of our main theorem. 
	
	\section{Preliminaries}
	
	In this section we lay out several fundamental lemmas to our later development. 
	
	With our first lemma we collect several elementary inequalities that we will ,utilize in our later development.
		\begin{lemma}\label{inequalities}
		For $x,y \in \mathbb{R}^N$ and $a,b \in \mathbb{R}^{+}$, we have the inequalities,
		\begin{align}
			\left|x\right|^{p-2} x \cdot\left(x-y\right) \geq \frac{1}{p} \left(\left|x\right|^p - \left|y\right|^p \right); \nonumber \\
			ab \leq \varepsilon a^p + \frac{1}{\varepsilon^{\frac{q}{p}}}b^q \ \ \mbox{ for $\varepsilon > 0$ and $p,q >1$ with $\frac{1}{p} + \frac{1}{q} = 1$, } \nonumber 
		\end{align}
	\end{lemma}
	The proof of the last inequality is contained in \cite{LX}.
	
	\begin{lemma}
		If $f$ is an increasing function on $\mathbb{R}$ and $F$ an anti-derivative of $f$, then, 
		\begin{equation}
			f(s)(s-t) \geq F(s) - F(t), \ \ \mbox{ for all $s, t \in \mathbb{R}$. } \label{ineq}
		\end{equation}
	\end{lemma}

	\begin{lemma}\label{odenone}
		(a) For $p>2$ and $x,y \in \mathbb{R}^N$ we have the inequality,
		\begin{equation}
			\left(\left|x\right|^{p-2} x - \left|y\right|^{p-2} y \right) \cdot \left(x - y \right) \geq \frac{1}{2^{p-1}} \left| x-y \right|^p; \nonumber 
		\end{equation}
		For $1<p\leq 2$, and $x,y \in \mathbb{R}^N$, 
		\begin{equation}
			\left(1+\left|x\right|^2+\left|y\right|^2\right)^{\frac{2-p}{2}}\left(\left(\left|x\right|^{p-2} x - \left|y\right|^{p-2} y\right)\cdot\left(x-y\right)\right) \geq (p-1)\left|x-y\right|^2. \nonumber 
		\end{equation}
	\end{lemma}
	The proof of this lemma can be found in \cite{JOden}. 
	
	The existence assertion we will be providing comes via the Leray-Schauder Fixed-Point Theorem. \cite{GT}
	\begin{lemma}\label{LSThm}
		Let $T$ be a compact mapping of a Banach space $\mathbb{B}$ into itself, and suppose that there exists a constant $M$ such that,
		\begin{equation}
			\left\| x \right\|_{\mathbb{B}} \leq M
		\end{equation}
		for all $x \in \mathbb{B}$ and $\sigma \in \left[0,1\right]$ satisfying $x = \sigma T\left(x\right)$. Then $T$ has a fixed point. 
	\end{lemma}
	By a compact mapping we mean that $T$ is continuous, and maps bounded sets in $\mathbb{B}$ into precompact ones. 
	
	An essential point of our proof will rely upon upgrading a weakly convergent sequence to be a strongly convergent one. This is accomplished via the following lemma. 
	\begin{lemma}[Lions-Aubin]\label{la}
		Let $X_0, X$ and $X_1$ be three Banach spaces with $X_0 \subseteq X \subseteq X_1$. Suppose that $X_0$ is compactly embedded in $X$ and that $X$ is continuously embedded in $X_1$. For $1 \leq p, q \leq \infty$, let
		\begin{equation*}
			W=\{u\in L^{p}([0,T];X_{0}): \partial_t u\in L^{q}([0,T];X_{1})\}.
		\end{equation*}
		Then:
		\begin{enumerate}
			\item[\textup{(i)}] If $p  < \infty$, then the embedding of W into $L^p([0, T]; X)$ is compact.
			\item[\textup{(ii)}] If $p  = \infty$ and $q  >  1$, then the embedding of W into $C([0, T]; X)$ is compact. 
		\end{enumerate}
	\end{lemma}
	
	\section{Approximate Problems}
	
	In this section we study a system of approximate problems. To describe our approximate problems we first let $\tau \in \left(0,1\right)$. Then, to approximate $\nabla_z E\left(z\right)$, we define, 
	\begin{equation}
		\ft\left(s\right) = \left(s+\tau\right)^{\frac{p-2}{2}} + \bz\left(s+\tau\right)^{-\frac{1}{2}}. \label{ftdefn}
	\end{equation}
	Then we form the system of equations,
	\begin{align}
		-\mdiv\left(\left[M\left(\nabla u \right) + \tau I\right]\nabla v \right) + \tau v &= \frac{w-u}{\tau}, \ \ \mbox{ in $\Omega$, } \label{ap1} \\
		- \mdiv\left(\ft\left(\left|\nabla u \right|^2\right)\nabla u\right)  + \tau u &= v, \ \ \mbox{ in $\Omega$,} \label{ap2} \\
		\nabla u \cdot \nu = \nabla v \cdot \nu &= 0, \ \ \mbox{ on $\partial\Omega$ }, \label{ap3} 
	\end{align}
	where $w$ is a given function.
	
	To form our approximate problems, we first discretized in the time direction. Then we added in the terms $-\tau\Delta v$ and $\tau v$ to equation \eqref{ap1} to make the principle operator uniformly elliptic, and due to the boundary conditions. For similar reasons we also added the lower-order term $\tau u$ to equation \eqref{ap2}. 
	
	Our goal in this section is to detail the existence of solutions to \eqref{ap1}-\eqref{ap3}. We first give our definition of a weak solution. 
	\begin{defn} Assume that
		\begin{equation}\label{ws}
			w\in L^2(\Omega).
		\end{equation}
		We say that a pair of functions $\left(u,v\right)$ is a weak solution to \eqref{ap1}-\eqref{ap3} if the following holds;
		
		(D1) We have $u \in W^{1,p}\left(\Omega\right)$, and $v \in \wot$, 
		
		(D2) $u$ and $v$ solve the integral equations,
		\begin{align}
			&\int_{\Omega} \left[M\left(\nabla u \right) + \tau I\right] \nabla v \cdot \nabla \varphi dx + \tau \int_{\Omega} v \varphi dx = \int_{\Omega} \left(\frac{w-u}{\tau} \right)\varphi dx \nonumber \\
			&\int_{\Omega} \ft\left(\left|\nabla u \right|^2\right)\nabla u \cdot \nabla \psi dx + \tau \int_{\Omega} u \psi dx = \int_{\Omega} v \psi dx \nonumber 
			& \ \ \mbox{ for all $\left(\varphi,\psi\right) \in \left(\wot\right)^2$. } \nonumber 
		\end{align}
	\end{defn}

	We then have the following theorem for the existence of weak solutions. 
	
	\begin{thm}
		Suppose that $w \in L^2\left(\Omega\right)$ is a given function, and $\tau > 0$. Then there is a pair of functions $\left(u, v \right) \in \wot$ which are a weak solution to \eqref{ap1}-\eqref{ap3}. 
	\end{thm} 
	
	\begin{proof}
		For the proof we make use of the Leray-Schauder Theorem \ref{LSThm}. To do this, we define a mapping $\mathbb{B}$ from $L^2\left(\Omega\right)$ into itself in the following manner; given $\psi \in L^2\left(\Omega\right)$, we first define $u$ to be the unique weak solution to the problem, 
		\begin{align}
			-\mdiv\left(\ft\left(\left|\nabla u \right|^2\right)\nabla u\right)  + \tau u &= \psi, \ \ \mbox{ in $\Omega$, } \label{ap4} \\
			\nabla u \cdot \nu &= 0, \ \ \mbox{ on $\partial\Omega$. } \label{ap5} 
		\end{align}
		Note that $\ft$ is monotone, i.e. 
		\begin{equation}
			\left( \ft\left(\left|\xi\right|^2\right)\xi - \ft\left(\left|\eta\right|^2\right)\eta \right) \cdot \left(\xi - \eta \right) \geq 0, \ \ \mbox{ for all $\xi, \eta \in \mathbb{R}^2$. }
		\end{equation}
		This
		implies the existence of a unique weak solution $u \in W^{1,p}\left(\Omega\right)\cap L^2\left(\Omega\right)$ to the problem, \eqref{ap4}-\eqref{ap5}. \cite{JOden} Using this $u$, we then form the problem, 
		\begin{align}
			-\mdiv\left(\left[M\left(\nabla u \right) + \tau I\right]\nabla v \right) + \tau v &= \frac{w-u}{\tau}, \ \ \mbox{ in $\Omega$, } \label{ap6} \\
			\nabla v \cdot \nu &= 0, \ \ \mbox{ on $\partial \Omega$. } \label{ap7} 
		\end{align}
		Since $M\left(\nabla u \right) + \tau I$ is a uniformly elliptic coefficient matrix, we can then also conclude from the classical theory for elliptic equations \cite{GT} that there is a unique weak solution $v \in \wot \cap C^{\alpha}\left(\overline{\Omega}\right) $ to the problem \eqref{ap6}-\eqref{ap7} due to \eqref{ws} and the fact that $N=2$. We finally define our mapping as $\mathbb{B}\left(\psi\right) = v$. As the solutions to the problems \eqref{ap4}-\eqref{ap5} and \eqref{ap6}-\eqref{ap7} are unique, we have that $\mathbb{B}$ is well-defined. 
		
		Next, the Sobolev Embedding Theorem (see \cite{GT} chap. 7) asserts that, $\wot$ is compactly embedded in $L^2\left(\Omega\right)$. From this we can conclude that our mapping $\mathbf{B}$ takes bounded sets in $L^2\left(\Omega\right)$ into precompact ones. 
		
		We will now move on to showing that $\mathbf{B}$ is continuous on $L^2\left(\Omega\right)$. To do so we first let $\{\psi_n\}$ be a sequence in $L^2\left(\Omega\right)$, and $\psi \in L^2\left(\Omega\right)$ so that, 
		\begin{equation}
			\psi_n \rightarrow \psi \ \ \mbox{ strongly in $L^2\left(\Omega\right)$. } \label{c1}
		\end{equation}
		
		Then, for each $n=1,2,3,...$ we set $v_n = \mathbf{B}\left(\psi_n\right)$. i.e. 
		\begin{align}
			-\mdiv\left(\left[M\left(\nabla u_n \right) + \tau I \right] \nabla v_n \right) + \tau v_n &= \frac{w-u_n}{\tau}, \ \ \mbox{ in $\Omega$, } \label{c2} \\
			-\mdiv\left(\ft\left(\left|\nabla u_n \right|^2\right)\nabla u_n\right)   + \tau u_n & = \psi_n, \ \ \mbox{ in $\Omega$, } \label{c3} \\
			\nabla u_n \cdot \nu = \nabla v_n \cdot \nu &= 0, \ \ \mbox{ on $\partial\Omega$. } \label{c4} 
		\end{align}
		
		We now use $u_n$ as a test function in equation \eqref{c3} to derive, 
		\begin{equation}
			\int_{\Omega} \ft\left(\left|\nabla u_n\right|^2\right) \left|\nabla u_n \right|^2 dx + \int_{\Omega} u_n^2 dx \leq c(\tau) \int_{\Omega} \psi_n^2 dx. \label{c5}
		\end{equation}
		
		As a consequence we may conclude that the sequence $\{u_n\}$ is uniformly bounded in $W^{1,p}\left(\Omega\right)\cap L^2\left(\Omega\right)$. Next, we use $v_n$ as a test function in \eqref{c2} to find, 
		\begin{equation}
			\int_{\Omega} M\left(\nabla u_n \right) \nabla v_n \cdot \nabla v_n dx + \tau \int_{\Omega} \left|\nabla v_n \right|^2 dx + \tau \int_{\Omega} v_n^2 dx = \int_{\Omega} \left(\frac{w-u_n}{\tau} \right) v_n dx. \label{c6}
		\end{equation}
		For the first integral on the left-hand side of \eqref{c6} we then have, 
		\begin{equation}
			\int_{\Omega} M\left(\nabla u_n \right) \nabla v_n \cdot \nabla v_n dx \geq \int_{\Omega} \frac{\left|\nabla v_n \right|^2 }{ 1+ q \left|\nabla u_n \right| } dx \geq 0. \label{c7}
		\end{equation}
		Then, for the integral on the right-hand side of \eqref{c6} we use H\"older's inequality to get,
		\begin{equation}
			\int_{\Omega} \left(\frac{w-u_n}{\tau} \right) v_n dx \leq \varepsilon \int_{\Omega} v_n^2 dx + c\left(\varepsilon\right) \int_{\Omega} \left| \frac{w-u_n}{\tau} \right|^2 dx. \label{c8}
		\end{equation}
		Upon using \eqref{ws}, \eqref{c7}, and \eqref{c8} in \eqref{c6}, and keeping in mind \eqref{c5} we then obtain, 
		\begin{equation}
			\int_{\Omega} \left|\nabla v_n \right|^2 dx + \int_{\Omega} v_n^2 dx \leq c(\tau) \label{c9} 
		\end{equation}
		With this we can conclude that the sequence $\{v_n\}$ is uniformly bounded in $\wot$. As such, there is a subsequence, which we won't relabel, so that, 
		\begin{align}
			u_n \rightarrow u& \ \ \mbox{ weakly in $W^{1,p}\left(\Omega\right)$, and weakly in $L^2\left(\Omega\right)$ } \label{c10} \\
			v_n \rightarrow v& \ \ \mbox{ weakly in $\wot$. } \label{c11}
		\end{align}
	Now from \eqref{c3} we derive,
		\begin{align}
			-\mdiv\left( \ft\left(\left|\nabla u_{n_1} \right|^2 \right)\nabla u_{n_1} - \ft\left(\left|\nabla u_{n_2} \right|^2 \right)\nabla u_{n_2} \right) + \tau\left( u_{n_1} - u_{n_2} \right) = \psi_{n_1} - \psi_{n_2}, \label{c12}
		\end{align}
	We use the function $\left(u_{n_1} - u_{n_2} \right)$ as a test function in $\eqref{c12}$ to find, 
		\begin{align}
			\int_{\Omega}& \left(\ft\left(\left|\nabla u_{n_1}\right|^2 \right) \nabla u_{n_1} - \ft\left(\left|\nabla u_{n_2} \right|^2 \right)\nabla u_{n_2} \right) \cdot \left( \nabla u_{n_1} - \nabla u_{n_2} \right) dx \nonumber \\
			& + \tau \int_{\Omega} \left( u_{n_1} - u_{n_2} \right)^2 dx = \int_{\Omega} \left(\psi_{n_1} - \psi_{n_2} \right) \left( u_{n_1} - u_{n_2} \right) dx \label{c13}
		\end{align}
		Then using \ref{inequalities} we have two cases. For $p > 2 $, we have,
		\begin{align}
			\int_{\Omega}& \left(\ft\left(\left|\nabla u_{n_1}\right|^2 \right) \nabla u_{n_1} - \ft\left(\left|\nabla u_{n_2} \right|^2 \right)\nabla u_{n_2} \right) \cdot \left( \nabla u_{n_1} - \nabla u_{n_2} \right) dx \geq \frac{1}{2^{p-1}} \int_{\Omega} \left| \nabla u_{n_1} - \nabla u_{n_2} \right|^p dx \label{c14} 
		\end{align}
		Then, for $1 < p \leq 2$ we use the second inequality in \ref{odenone} to estimate;
		\begin{align}
			\int_{\Omega} \left|\nabla u_{n_1} - \nabla u_{n_2} \right|^p dx &= \int_{\Omega} \left( 1 + \left|\nabla u_{n_1} \right|^2 + \left|\nabla u_{n_2} \right|^2 \right)^{\frac{p(2-p)}{4}} \frac{ \left|\nabla u_{n_1} - \nabla u_{n_2} \right|^p }{\left( 1 + \left|\nabla u_{n_1} \right|^2 + \left|\nabla u_{n_2} \right|^2 \right)^{\frac{p(2-p)}{4}} } dx \nonumber \\
			& \leq \left( \int_{\Omega} \left(1+ \left|\nabla u_{n_1}\right|^2 + \left|\nabla u_{n_2} \right|^2 \right)^{\frac{p-2}{2}} \left|\nabla u_{n_1} - \nabla u_{n_2} \right|^2 dx \right)^{\frac{p}{2}} \nonumber \\
			& \cdot \left( \int_{\Omega} \left( 1+ \left|\nabla u_{n_1}\right|^2 +\left| \nabla u_{n_2} \right|^2 \right)^{\frac{p}{2}} dx \right)^{\frac{2-p}{2} } \nonumber \\
			& \leq c \left[ \int_{\Omega} \left( \ft\left(\left|\nabla u_{n_1} \right|^2 \right) \nabla u_{n_1} - \ft\left(\left|\nabla u_{n_2} \right|^2 \right) \nabla u_{n_2} \right) \cdot \left( \nabla u_{n_1} - \nabla u_{n_2} \right) dx \right]^{\frac{p}{2}}. \label{c15}
		\end{align}

		As the sequence $\psi_n$ converges strongly in $L^2\left(\Omega\right)$ and $u_n$ converges weakly in $L^2\left(\Omega\right)$, we deduce that, 
		\begin{equation}
			\int_{\Omega} \left(\psi_n - \psi \right) \left( u_n - u \right) dx \rightarrow 0. \nonumber 
		\end{equation}
		Using this along with \eqref{c14} and \eqref{c15} in \eqref{c13} we then derive for all $1< p < \infty$,
		\begin{equation}
			\int_{\Omega} \left|\nabla u_{n_1} - \nabla u_{n_2} \right|^p dx \rightarrow 0, \ \ \mbox{ as $n_1, n_2 \rightarrow \infty$. } \label{c16}
		\end{equation}
		Consequently we obtain, at least for a subsequence, 
		\begin{align}
			u_n \rightarrow &u, \ \ \mbox{ strongly in $\wot$, } \nonumber \\
			u_n \rightarrow &u, \ \ \mbox{ pointwise a.e. in $\Omega$, } \nonumber \\
			\nabla u_n \rightarrow &\nabla u, \ \ \mbox{ pointwise a.e. in $\Omega$, } \nonumber \\
			M\left(\nabla u_n \right) \rightarrow &M\left( \nabla u \right), \ \ \mbox{strongly in   $L^p(\Omega)$ for each $p>1$. } \nonumber 
		\end{align}
		The last convergence can easily be verified from our definition of $M(\nabla u)
		$ in the introduction. From the last of the these, we can then use the weak convergence of $\nabla v_n$ along with the strong convergence of $M\left(\nabla u_n \right)$, to conclude that, 
		\begin{equation}
			M\left(\nabla u_n \right) \nabla v_n \rightarrow M\left(\nabla u \right) \nabla v, \ \ \mbox{ weakly in $L^2\left(\Omega\right)$. } \nonumber 
		\end{equation}
		We are then able to pass to the limit in \eqref{c2} to find, 
		\begin{align}
			-\mdiv\left(\left[M\left(\nabla u \right) + \tau I \right] \nabla v \right) + \tau v &= \frac{w-u}{\tau}, \ \ \mbox{ in $\Omega$, } \label{c17} \\
			\nabla v \cdot \nu &= 0, \ \ \mbox{ on $\partial\Omega$. } \label{c18} 
		\end{align}
		We then subtract \eqref{c17} from \eqref{c2}, 
		\begin{equation}
			-\mdiv\left( M\left(\nabla u_n\right)\nabla v_n - M\left(\nabla u\right)\nabla v \right) - \tau \Delta \left(v_n -v \right) + \tau \left(v_n - v\right) = \frac{u-u_n}{\tau}, \ \ \mbox{ in $\Omega$. } \label{c19} 
		\end{equation}
		We then use $\left(v_n - v \right)$ as a test function in $\eqref{c19}$ and obtain, 
		\begin{align}
			&\int_{\Omega} \left( M\left(\nabla u_n\right)\nabla v_n - M\left(\nabla u\right)\nabla v \right) \cdot \nabla \left(v_n - v \right) dx + \tau \int_{\Omega} \left|\nabla \left( v_n - v \right) \right|^2 dx \nonumber \\
			& + \tau \int_{\Omega} \left( v_n - v \right)^2 dx = \int_{\Omega} \left( \frac{u-u_n}{\tau} \right) \left(v_n - v \right) dx. \label{c20}
		\end{align}
		We then have, 
		\begin{align}
			\int_{\Omega} \left( M\left(\nabla u_n\right)\nabla v_n - M\left(\nabla u\right)\nabla v \right) \cdot \nabla \left(v_n - v \right) dx &= \int_{\Omega}  M\left(\nabla u_n\right) \nabla (v_n-v)\cdot \nabla (v_n - v) dx \nonumber \\
			&+ \int_{\Omega} \left(M\left(\nabla u_n\right)-M\left(\nabla u \right)\right) \nabla v \cdot \nabla \left(v_n - v \right) dx \nonumber \\
			& \geq \int_{\Omega} \left(M\left(\nabla u_n\right)-M\left(\nabla u \right)\right) \nabla v \cdot \nabla \left(v_n - v \right) dx \nonumber \\
			& \rightarrow 0. 
		\end{align}
		Then, for the right-hand side we use the strong convergence of $\{u_n\}$ and the weak convergence of $v_n$ to obtain, 
		\begin{equation}
			\int_{\Omega} \left( \frac{u-u_n}{\tau} \right) \left(v_n - v \right) dx \rightarrow 0. \nonumber 
		\end{equation}
		Consequently we can conclude that $\{v_n\}$ converges strongly in $\wot$ to $v$. 
		
		Thus we have shown that every subsequence of $\{\mathbf{B}\left(\psi_n\right)\}$ has a further convergent subsequence, all of which converge to the same limit $\mathbf{B}\left(\psi\right)$. Therefore we can conclude that the whole sequence must converge to $\mathbf{B}\left(\psi\right)$. That is to say, $\mathbf{B}$ is a continuous mapping on $L^2\left(\Omega\right)$. 
		
		There is one final condition left to check in order to verify all of the hypothesis of the Leray-Schauder Theorem. \ref{LSThm} Suppose that $\sigma \in \left(0,1\right)$ and that $v\in L^2\left(\Omega\right)$ satisfies, $v= \sigma \mathbf{B}\left(v\right)$. This equation is equivalent to the system, 
		\begin{align}
			-\mdiv\left(\left[M\left(\nabla u \right) + \tau I \right] \nabla v \right) + \tau v &= \sigma \left(\frac{w-u}{\tau} \right), \ \ \mbox{ in $\Omega$, } \label{lsf1} \\
			-\mdiv\left(\ft\left(\left|\nabla u\right|^2\right)\nabla u\right) + \tau u &= v, \ \ \mbox{ in $\Omega$, } \label{lsf2} \\
			\nabla u \cdot \nu = \nabla v \cdot \nu &= 0, \ \ \mbox{ on $\partial\Omega$. } \label{lsf3}
		\end{align}
		We then use $v$ as a test function in \eqref{lsf1} to derive, 
		\begin{equation}
			\int_{\Omega} \left[M\left(\nabla u \right) + \tau I\right] \nabla v \cdot \nabla v dx + \tau \int_{\Omega} v^2 dx  + \frac{\sigma}{\tau} \int_{\Omega} u v dx = \frac{\sigma}{\tau} \int_{\Omega} w v dx. \label{lsf4} 
		\end{equation}
		For the first integral on the left-hand side we have, 
		\begin{equation}
			\int_{\Omega} \left[M\left(\nabla u \right) + \tau I\right] \nabla v \cdot \nabla v dx \geq \int_{\Omega} \frac{\left|\nabla v \right|^2 }{1+q\left|\nabla u \right|} dx + \tau \int_{\Omega} \left|\nabla v \right|^2 dx. \label{lsf5} 
		\end{equation}
		Next, for the final integral on the left-hand side of \eqref{lsf4} we use \eqref{lsf2} to calculate,
		\begin{align}
			\int_{\Omega} u v dx &= - \int_{\Omega} u \mdiv\left( \ft\left(\left|\nabla u \right|^2\right)\nabla u \right) dx + \tau \int_{\Omega} u^2 dx \nonumber \\
			& = \int_{\Omega} \ft\left(\left|\nabla u\right|^2\right)\left|\nabla u \right|^2 dx + \tau \int_{\Omega} u^2 dx \nonumber \\
			& \geq 0. \label{lsf6}
		\end{align}
		We then use Young's inequality on the right-hand side of \eqref{lsf4} to estimate, 
		\begin{equation}
			\frac{\sigma}{\tau} \int_{\Omega} w v dx \leq \varepsilon \int_{\Omega} v^2 dx + \frac{c(\varepsilon)}{\tau^2} \int_{\Omega} w^2 dx. \label{lsf7}
		\end{equation}
		We use \eqref{lsf5}-\eqref{lsf7} in \eqref{lsf4}, and choose $\varepsilon$ to be sufficiently small then yields, 
		\begin{equation}
			\int_{\Omega} \left|\nabla v \right|^2 dx + \int_{\Omega} v^2 dx \leq c\left(\tau\right) \int_{\Omega} w^2 dx. \label{lsf8}
		\end{equation}
		With \eqref{lsf8} we have the final piece required for the Leray-Schauder Fixed-point Theorem. Obviously such a fixed point is a weak solution to \eqref{ap1}-\eqref{ap3}. The proof is complete. 
	\end{proof}
	
	\section{Proof of the Main Theorem}
	
	In this section we give our proof of the main theorem. We first present a time discretized problem, which is based on \eqref{ap1}-\eqref{ap3}. We then derive a-priori estimates for our approximate solutions. In the final step we then prove some compactness results that allow us to justify passing to the limit. 
	
	We begin by describing the time descretized problem. We first let $T>0$ be given. Then for each $j = 1,2,3,...$ we divide the time interval $\left[0,T\right]$ into $j$ sub-intervals of equal size. Next set, 
	\begin{equation}
		\tau = \frac{T}{j}, \ \ \mbox{ and } \ \ t_k = k\tau, \ k=0,1,2,...,j. \nonumber 
	\end{equation}
	Then for each $k$, we recursively solve the system, 
	\begin{align}
		\frac{u_k - u_{k-1} }{\tau } - &\mdiv\left(\left[M\left(\nabla u_k \right) + \tau I \right] \nabla v_k \right) + \tau v_k = 0, \ \ \mbox{ in $\Omega$, }  \label{td1} \\
		- & \mdiv\left(\ft\left(\left|\nabla u_k \right|^2 \right) \nabla u_k \right) + \tau u_k = v_k, \ \ \mbox{ in $\Omega$, } \label{td2} \\
		&\nabla u_k \cdot \nu = \nabla v_k \cdot \nu = 0, \ \ \mbox{ on $\partial\Omega$. } \label{td3}
	\end{align}
	
	Now we introduce the functions, 
	\begin{align}
		\tuj &= \left(\frac{t-t_{k-1}}{\tau}\right) u_k\left(x\right) - \left(1 - \frac{t-t_{k-1}}{\tau} \right) u_{k-1}, \ \ \mbox{ on $(t_{k-1},t_k]\times \Omega$, } \label{tujdefn} \\
		\buj &= u_k, \ \ \mbox{ on $(t_{k-1},t_k]\times \Omega$, } \label{bujdefn} \\
		\bvj &= u_k, \ \ \mbox{ on $(t_{k-1},t_k]\times \Omega$, } \label{bvjdefn} 
	\end{align}
	
	Using these functions we may express our time-discrete problems \eqref{td1}-\eqref{td3} as, 
	\begin{align}
		\partial_t \tuj &- \mdiv\left( \left[ M\left(\nabla \buj \right) +\tau I\right] \nabla \bvj \right) + \tau \bvj = 0, \ \ \mbox{ on $\ot$, } \label{td4} \\
		&-\mdiv\left(\ft\left(\left|\nabla \buj \right|^2 \right)\nabla \buj \right) + \tau \buj = \bvj, \ \ \mbox{ on $\ot$, } \label{td5} 
	\end{align}
	
	Our next goal is to the discrete analogue of our a priori estimate \eqref{I9}. This is our basic starting point for being able to justify passing to the limit in \eqref{td4}-\eqref{td5}. 
	\begin{prop}\label{apdiscrete}
		There exists a constant $c$ which is independent of $\tau$ and depends only on the given data so that,  
		\begin{align}
			\max_{0\leq t \leq T} &\left( \int_{\Omega} \left( \left|\nabla \buj \right|^p + \left|\nabla \buj\right| \right) dx + \tau \int_{\Omega} \left|\buj \right|^2 dx \right) + \int_{\ot} \frac{ \left|\nabla \bvj \right|^2 }{1 + q \left|\nabla \buj \right| } dxdt \nonumber  \\
			& + \tau \int_{\ot} \left|\nabla \bvj \right|^2 dx + \tau \int_{\ot} \left|\bvj \right|^2 dxdt \nonumber \\
			& \leq c \left( \int_{\Omega} \left(\left|\nabla u_0 \right|^p + \left|\nabla u_0 \right| \right) dx + \tau \int_{\Omega} \left|u_0\right|^2 dx \right) \label{apriori1}
		\end{align}
	\end{prop}
	\begin{proof}
		For the proof we use $v_k$ as a test function in \eqref{td1}. 
		\begin{equation}
			\frac{1}{\tau} \int_{\Omega} \left(u_k - u_{k-1}\right) v_k dx + \int_{\Omega} \left[M\left(\nabla u_k\right) + \tau I \right] \nabla v_k \cdot \nabla v_k dx + \tau \int_{\Omega} v_k^2 dx = 0. \label{po1}
		\end{equation}
		Then for the first integral on the left-hand side of \eqref{po1} we use \eqref{td2} to calculate, 
		\begin{align}
			\frac{1}{\tau} \int_{\Omega} \left(u_k - u_{k-1}\right) v_k dx &= -\frac{1}{\tau} \int_{\Omega} \left(u_k - u_{k-1}\right)\mdiv\left(\ft\left(\left|\nabla u_k \right|^2 \right)\nabla u_k \right) dx + \int_{\Omega} \left(u_k - u_{k-1} \right) u_k dx \nonumber \\
			&= \frac{1}{\tau} \int_{\Omega} \nabla \left(u_k - u_{k-1} \right) \cdot \ft\left(\left|\nabla u_k \right|^2 \right) \nabla u_k dx + \int_{\Omega} \left(u_k - u_{k-1} \right) u_k dx \nonumber \\
			& \geq \frac{1}{p} \int_{\Omega} \left[ \left( \left|\nabla u_k\right|^2 + \tau \right)^{\frac{p}{2}} - \left( \left|\nabla u_{k-1}\right|^2 + \tau \right)^{\frac{p}{2}} \right] dx \nonumber \\
			&+ \bz \int_{\Omega} \left[ \left(\left|\nabla u_k \right|^2 + \tau \right)^{\frac{1}{2}} - \left(\left|\nabla u_{k-1}\right|^2 + \tau\right)^{-\frac{1}{2} } \right] dx \label{po2}
		\end{align}
		Where in the last line we have used \ref{inequalities}. 
		
		For the second integral on the left-hand side of \eqref{po1} we have, 
		\begin{align}
			\int_{\Omega} \left[M\left(\nabla u_k\right) + \tau I \right] \nabla v_k \cdot \nabla v_k dx &= \int_{\Omega} M\left(\nabla u_k \right) \nabla v_k \cdot \nabla v_k dx + \tau \int_{\Omega} \left|\nabla v_k \right|^2 dx \nonumber \\
			& \geq \int_{\Omega} \frac{ \left|\nabla v_k \right|^2 }{ 1 + q \left|\nabla u_k \right| } dx + \tau \int_{\Omega} \left|\nabla v_k \right|^2 dx. \label{po3} 
		\end{align}
		Putting \eqref{po2} and \eqref{po3} back into \eqref{po1} then yields, 
		\begin{align}
			\frac{1}{p\tau}&  \int_{\Omega} \left[ \left( \left|\nabla u_k\right|^2 + \tau \right)^{\frac{p}{2}} - \left( \left|\nabla u_{k-1}\right|^2 + \tau \right)^{\frac{p}{2}} \right] dx + \bz \int_{\Omega} \left[ \left(\left|\nabla u_k \right|^2 + \tau \right)^{\frac{1}{2}} - \left(\left|\nabla u_{k-1}\right|^2 + \tau\right)^{-\frac{1}{2} } \right] dx \nonumber \\
			& + \frac{1}{2} \int_{\Omega} \left(u_k^2 - u_{k-1}^2 \right) dx + \int_{\Omega} \frac{ \left|\nabla v_k \right|^2 }{ 1 + q \left|\nabla u_k \right| } dx + \tau \int_{\Omega} \left|\nabla v_k \right|^2 dx + \tau \int_{\Omega} v_k^2 dx \leq 0. \nonumber 
		\end{align}
		We then multiply through this equation by $\tau$, and then sum the result over $k$ to derive \eqref{apriori1}.
	\end{proof}
	
	With this estimate we now begin proving uniform bounds in explicit function spaces for our sequences of approximate solutions. This is done in the next two propositions.
	\begin{prop}\label{uspaces}
		The sequences $\{\tuj\}$ and $\{\buj\}$ are each uniformly bounded in the space $L^{\infty}\left(0,T; W^{1,p}\left(\Omega\right) \right)$. 
	\end{prop}
	\begin{proof}
		Continuing, we first multiply through \eqref{td5} by $\tau$ and then add the resulting equation to \eqref{td4} to obtain, 
		\begin{equation}
			\partial_t \tuj = \mdiv\left(\left[M\left(\nabla \buj\right) + \tau I \right] \nabla \bvj\right) + \tau \mdiv\left(\ft\left(\left|\nabla \buj\right|^2 \right) \nabla \buj \right) - \tau^2 \buj  \nonumber 
		\end{equation}
		
		We then use this equation to calculate, 
		\begin{align}
			\frac{d}{dt} \int_{\Omega} \tuj dx &= \int_{\Omega} \partial \tuj dx \nonumber \\
			& = \int_{\Omega} \mdiv\left(\left[M\left(\nabla \buj\right) + \tau I \right] \nabla \bvj\right) dx \nonumber \\
			& +\tau \int_{\Omega} \mdiv\left(\ft\left(\left|\nabla \buj\right|^2 \right) \nabla \buj \right) dx - \tau^2 \int_{\Omega} \buj dx \nonumber \\
			& = - \tau^2 \int_{\Omega} \buj dx  . \label{sp1}
		\end{align}
		Upon integrating \eqref{sp1} with respect to $t$ and utilizing prop \ref{apriori1} we then derive
		\begin{equation}
			\max_{0\leq t \leq T} \left| \int_{\Omega} \tuj dx \right| \leq c. \label{sp2}
		\end{equation}
		Then we use Poincare's inequality along with \eqref{sp2} to estimate, 
		\begin{align}
			\int_{\Omega} \left|\tuj\right|^p dx & \leq c \int_{\Omega} \left|\tuj - \frac{1}{\left|\Omega\right|} \int_{\Omega} \tuj dx \right|^2 dx + c \left| \int_{\Omega} \tuj dx \right|^p \nonumber \\
			& \leq c \int_{\Omega} \left|\nabla \tuj \right|^p dx + c. \label{sp3}
		\end{align}
		Then, to estimate the gradient of $\tuj$, we first choose any $t \in \left(0,T\right]$. Then there exists a number $k$ so that, $t\in \left(t_{k-1} - t_{k} \right]$. Subsequently, we have, 
		\begin{align}
			\int_{\Omega} \left|\nabla \tuj(x,t)\right|^p dx & = \int_{\Omega} \left| \frac{t-t_{k-1}}{\tau} \nabla u_k + \left(1-\frac{t-t_{k-1}}{\tau}\right) \nabla u_{k-1} \right|^p dx \nonumber \\
			& \leq \frac{t-t_{k-1}}{\tau} \int_{\Omega} \left|\nabla u_k \right|^p dx + \left(1-\frac{t-t_{k-1}}{\tau}\right) \int_{\Omega} \left|\nabla u_{k-1} \right|^p dx \nonumber \\
			& \leq c \sup_{0\leq t \leq T} \int_{\Omega} \left|\nabla \buj \right|^p dx \leq c. \label{sp4}
		\end{align}
		Combining this with \eqref{sp3} we can then conclude that the sequence $\{\tuj\}$ is uniformly bounded in the space $L^{\infty}\left(0,T; W^{1,p}\left(\Omega\right) \right)$. It is also clear that, 
		\begin{equation}
			\sup_{0\leq t \leq T}\int_{\Omega} \left|\buj\right|^p dx \leq c \sup_{0\leq t \leq T} \int_{\Omega} \left|\tuj\right|^p dx. \nonumber 
		\end{equation}
		Consequently we also know that $\{\buj\}$ is uniformly bounded in the space $L^{\infty}\left(0,T; W^{1,p}\left(\Omega\right) \right)$.
	\end{proof}
	
	\begin{prop}\label{vspace}
		The sequence $\{\bvj\}$ uniformly bounded in the spaces, $L^{2}\left(0,T;W^{1,\frac{2p}{p+1}}\left(\Omega\right)\right)$ and  $L^2\left(0,T; L^{2p}\left(\Omega\right)\right)$. 
	\end{prop}
	\begin{proof}
		For the proof we first use H\"older's inequality, to estimate, 
		\begin{align}
			\left\|\nabla \bvj \right\|_{\frac{2p}{P+1}, \Omega} &= \left(\int_{\Omega} \left(1+q\left|\nabla u \right| \right)^{\frac{p}{p+1}} \frac{\left|\nabla \bvj \right|^{\frac{2p}{p+1} } }{\left(1+q\left|\nabla \buj\right|\right)^{\frac{p}{p+1}} } dx \right)^{\frac{p+1}{2p} } \nonumber \\
			& \leq \left\| 1+q\left|\nabla \buj \right|\right\|^{\frac{1}{2}}_{p,\Omega} \left( \int_{\Omega} \frac{\left|\nabla \bvj \right|^2 }{1+q\left|\nabla \buj \right| } dx \right)^{\frac{1}{2}}  \label{vs1} 
		\end{align}
		Now square each side of \eqref{vs1} and integrate with respect to $t$ and use proposition \ref{apriori1} to obtain, 
		\begin{align}
			\left\|\nabla \bvj \right\|_{L^2\left(0,T;L^{\frac{2p}{p+1}}\left(\Omega\right) \right) } & \leq \sup_{0 \leq t \leq T } \left\| 1+q\left|\nabla\buj\right| \right\|_{p,\Omega} \int_{\ot} \frac{ \left|\nabla \bvj \right|^2 }{1 + q\left|\nabla \buj\right| } dxdt \leq c. \label{z}
		\end{align}
		Next, we integrate \eqref{td5} over $\Omega$ to find, 
		\begin{align}
			\int_{\Omega} \bvj dx &= - \int_{\Omega} \mdiv\left(\ft\left(\left|\nabla\buj\right|^2\right)\nabla\buj\right) dx + \tau \int_{\Omega} \buj dx \nonumber \\
			& = \tau \int_{\Omega} \buj dx. \label{vs2}
		\end{align}
		Then we use \ref{apriori1} to conclude from \eqref{z} that, 
		\begin{equation}
			\max_{0\leq t \leq T} \left| \int_{\Omega} \bvj dx \right| \leq c. \label{vs3} 
		\end{equation}
		We then use Poincare's inequality and \eqref{vs3} to obtain, 
		\begin{align}
			\left\|\bvj \right\|_{2p,\Omega} & \leq c \left\| \bvj - \frac{1}{\left|\Omega\right|} \int_{\Omega} \bvj dx \right\|_{2p,\Omega} + c\left| \int_{\Omega} \bvj dx \right| \nonumber \\
			& \leq c \left\| \nabla \bvj \right\|_{\frac{2p}{p+1}, \Omega} + c. \label{vs4}
		\end{align}
		Square each side of \eqref{vs4} then integrate with respect to $t$, and utilize \eqref{z} to find,
		\begin{equation}
			\left\| \bvj \right\|_{L^2\left(0,T; L^{2p}\left(\Omega\right)\right)} \leq c. \label{vs5} 
		\end{equation}
		With this we can conclude the proof. 
	\end{proof}
	
	On account of propositions \ref{uspaces} and \ref{vspace} we can conclude that there are at least subsequences, which we will not relabel, which converge weakly in there respective spaces. Obviously these are not enough to justify passing to the limit in our equations. Due to the nonlinearity present in \eqref{td4} we must have the strong convergence of $\{\nabla \buj\}$. This will guarantee a subsequence that converges point wise a.e. on $\ot$. Without this condition we would be unable to prove that the coefficient matrix $M\left(\nabla \buj\right)$ converges to $M\left(\nabla u\right)$ in a meaningful way. In the next few propositions we establish the necessary compactness.
	
	\begin{prop}\label{ucomp}
		Passing to appropriate sub-sequences where necessary, $\{\tuj\}$ and $\{\buj\}$ both converge strongly in $C\left(0,T;L^p\left(\Omega\right)\right)$, and have the same limit $u$.
	\end{prop}
	\begin{proof}
		First notice that from proposition \ref{uspaces} we have that $\{\tuj\}$ is uniformly bounded in $L^{\infty}\left(0,T;W^{1,p}\left(\Omega\right) \right)$. Then, let $\varphi$ be any function in the space $W^{1,\frac{2p}{p-1}}\left(\Omega\right)$. We can then estimate with the aid of proposition \ref{vspace}, and \eqref{td4},
		\begin{align}
			\langle \partial_t \tuj, \varphi \rangle &= \int_{\Omega} \partial_t \tuj \varphi dx \nonumber \\
			&= \int_{\Omega} \mdiv\left(\left[M\left(\nabla \buj \right) + \tau I\right] \nabla \bvj \right) \varphi dx - \tau \int_{\Omega} \bvj \varphi dx \nonumber \\
			&= - \int_{\Omega} \left[ M\left(\nabla \buj \right) + \tau I \right] \nabla \bvj \cdot \nabla \varphi dx - \tau \int_{\Omega} \bvj \varphi dx \nonumber \\
			& \leq c \left\| \nabla \bvj \right\|_{\frac{2p}{p+1}, \Omega} \left\| \nabla \varphi \right\|_{\frac{2p}{p-1}, \Omega} + c \left\| \bvj \right\|_{\frac{2p}{p+1},\Omega} \left\|\varphi \right\|_{\frac{2p}{p-1},\Omega} \nonumber \\
			& \leq c \left\| \bvj \right\|_{W^{1,\frac{2p}{p+1}}\left(\Omega\right)} \left\| \varphi \right\|_{W^{1,\frac{2p}{p-1}}\left(\Omega\right)}. \label{uc1} 
		\end{align}
		As a consequence of \eqref{uc1} we can then conclude that, 
		\begin{equation}
			\left\| \partial_t \tuj \right\|_{\left( W^{1,\frac{2p}{p-1}}\left(\Omega\right) \right)^*} \leq c \left\| \bvj \right\|_{W^{1,\frac{2p}{p+1}}\left(\Omega\right) }. \label{uc2} 
		\end{equation}
		We then square each side of \eqref{uc2} and then integrate the resulting inequality with respect to $t$ to derive, using proposition \ref{vspace},
		\begin{equation}
			\int_{0}^{T} \left\|\partial_t \tuj \right\|^2_{\left( W^{1,\frac{2p}{p-1}}\left(\Omega\right) \right)^*} dt \leq c. \label{uc3}
		\end{equation}
		That is to say, we have shown that $\partial_t \tuj$ is uniformly bounded in the space $L^{2} \left(0,T; \left(W^{1,\frac{2p}{p-1}}\left(\Omega\right) \right)^* \right)$, while $\{\tuj\}$ is uniformly bounded in $L^{\infty}\left(0,T;W^{1,p}\left(\Omega\right)\right)$. Then, $W^{1,p}\left(\Omega\right)$ is compactly embedded in the space $L^{p}\left(\Omega\right)$. On the other hand, for any $\xi \in L^p\left(\Omega\right)$ and $\varphi \in W^{1,\frac{2p}{p-1}}\left(\Omega\right)$ we have by H\'older's inequality, 
		\begin{align}
			\int_{\Omega} \xi \varphi dx & \leq \left\| \xi \right\|_{p,\Omega} \left\| \varphi \right\|_{\frac{p}{p-1},\Omega} \nonumber \\
			& \leq c \left\| \xi \right\|_{p,\Omega} \left\|\varphi \right\|_{\frac{2p}{p-1},\Omega}. \nonumber 
		\end{align}
		We can then conclude that $L^p\left(\Omega\right)$ is continuously embedded in $\left(W^{1,\frac{2p}{p-1}}\left(\Omega\right)\right)^*$. This then puts us in a position to be able to use the Aubin-Lions Lemma \ref{la}. Upon doing so we obtain that the sequence $\{\tuj\}$ is precompact in the space $C\left([0,T]; L^p\left(\Omega\right)\right)$.
		
		Next, to demonstrate that the sequence $\{\buj\}$ is precompact as well it is enough to show it has the same limit as $\{\tuj\}$. Towards this end we first compute, 
		\begin{align*}
			\int_{0}^{T} \left(\tuj-\buj\right) dt &= \sum_{k=1}^{j} \left[ \int_{t_{k-1}}^{t_k} \frac{t-t_{k-1}}{\tau} \left(u_k-u_{k-1}\right) - \left(u_k-u_{k-1}\right) \right] dt \\
			&= -\frac{1}{2} \tau \sum_{k=1}^{j} \left(u_k-u_{k-1}\right) \\
			& = -\frac{1}{2} \tau \left( u_j - u_0\right).
		\end{align*}
		Now upon integrating the above equation over $\Omega$ and using \ref{uspaces} we find,
		\begin{align*}
			\int_{\ot} \left(\tuj-\buj\right) dx &\leq c \tau \sup_{0\leq t \leq T} \int_{\Omega} \tuj dx \\
			& \leq c\tau. 
		\end{align*} 
		This completes the proof.  
	\end{proof}
	
	Due to the nonlinear term $M\left(\nabla \buj \right) \nabla \bvj$ we will need further the strong convergence of the gradients $\{\nabla \buj\}$. This is done in the next proposition. 
	
	\begin{prop}\label{nablaucomp}
		The sequence $\{\nabla \buj\}$ is precompact in $\left(L^p\left(\ot\right)\right)^2$. 
	\end{prop}
	\begin{proof}
		For the proof, we first derive from \eqref{td5} the equation,
		\begin{align}
			- \mdiv\left( F_{\tau_1}\left(\left|\nabla \bar{u}_{j_1} \right|^2\right)\nabla \bar{u}_{j_1} - F_{\tau_2}\left(\left|\nabla \bar{u}_{j_2} \right|^2\right)\nabla \bar{u}_{j_2} \right) + \tau_1 \bar{u}_{j_1} - \tau_2 \bar{u}_{j_2} = \bar{v}_{j_1} - \bar{v}_{j_1}, \ \ \mbox{ in $\ot$. } \label{nu1}
		\end{align}
		Next, we use $\left(\buj - u \right)$ as a test function above to find, 
		\begin{align}
			\int_{\ot}& \left[  F_{\tau_1}\left(\left|\nabla \bar{u}_{j_1} \right|^2\right)\nabla \bar{u}_{j_1} - F_{\tau_2}\left(\left|\nabla \bar{u}_{j_2} \right|^2\right)\nabla \bar{u}_{j_2} \right] \cdot \left( \nabla \bar{u}_{j_1} - \nabla \bar{u}_{j_2} \right) dxdt \nonumber \\
			+ &\int_{\ot} \left(\tau_1 \bar{u}_{j_1} - \tau_2 \bar{u}_{j_2}\right) \left( \bar{u}_{j_1} - \bar{u}_{j_2} \right) dxdt = \int_{\ot} \left( \bar{v}_{j_1} - \bar{v}_{j_2} \right)\left(\bar{u}_{j_1} - \bar{u}_{j_2} \right) dxdt \label{nu2}
		\end{align}
	
		As before there are now two cases. The first case we consider is that $p>2$. We then derive from lemma \ref{inequalities} that, 
		\begin{equation}
			\left( F_{\tau_1}\left(\left|\nabla \bar{u}_{j_1} \right|^2 \right) \nabla \bar{u}_{j_1} - F_{\tau_2}\left(\left|\nabla \bar{u}_{j_2} \right|^2\right) \nabla \bar{u}_{j_2} \right) \cdot \left( \nabla \bar{u}_{j_1} - \nabla \bar{u}_{j_2} \right) \geq c \left|\nabla \bar{u}_{j_1} - \nabla \bar{u}_{j_2} \right|^p, \label{nu3}
		\end{equation}
		We then use \eqref{nu3} in \eqref{nu2} to find, 
		\begin{equation}
			c \int_{\ot} \left|\nabla \bar{u}_{j_1} - \nabla \bar{u}_{j_2} \right|^p dxdt + \int_{\ot} \left(\tau_1 \bar{u}_{j_1} - \tau_2 \bar{u}_{j_2}\right) \left( \bar{u}_{j_1} - \bar{u}_{j_2} \right) dxdt = \int_{\ot} \left( \bar{v}_{j_1} - \bar{v}_{j_2} \right)\left(\bar{u}_{j_1} - \bar{u}_{j_2} \right)  dxdt \label{nu4}
		\end{equation}

		Then from lemma's \ref{ucomp} and \ref{apdiscrete} we can conclude that, 
		\begin{align}
			 \int_{\ot} \left(\tau_1 \bar{u}_{j_1} - \tau_2 \bar{u}_{j_2}\right) \left( \bar{u}_{j_1} - \bar{u}_{j_2} \right) dxdt & \leq \left\| \tau_1 \bar{u}_{j_1} - \tau_2 \bar{u}_{j_2}\right\|_{\frac{p}{p-1},\ot} \left\| \bar{u}_{j_1} - \bar{u}_{j_2} \right\|_{p,\ot} \nonumber \\
			 & \leq c \left\| \bar{u}_{j_1} - \bar{u}_{j_2} \right\|_{p,\ot} \nonumber \\
			 & \rightarrow 0, \ \ \mbox{ as $j\rightarrow \infty$. } \label{nu5}
		\end{align}
		Then for $p>2$, we have from \ref{vspace} that $\{\bar{v}_j\}$ is weakly convergent in the space $L^{\frac{p}{p-1}}\left(\ot\right)$ and then from \ref{ucomp} that $\{\buj\}$ converges strongly in $L^p\left(\ot\right)$. As a result the product of the two sequences converges weakly and we have, 
		\begin{equation}
			\int_{\ot} \left(\bvj - v \right)\left(\buj - u \right) dxdt \rightarrow 0, \ \ \mbox{ as $j\rightarrow \infty$. } \label{nu6}
		\end{equation}
		We can then conclude from \eqref{nu4}, \eqref{nu5} and \eqref{nu6} that, at least for a subsequence,
		\begin{equation}
			\nabla \buj \rightarrow \nabla u, \ \ \mbox{ strongly in $\left(L^p\left(\ot\right)\right)^{2}$ and a.e. on $\ot$. } \label{nu7}
		\end{equation}
		
		The second case is that $1<p < 2$. This time we use the second inequality in \ref{inequalities} to derive in an entirely similar way to \eqref{c15} 
		\begin{align}
			\int_{\ot} & \left|\nabla \bar{u}_{j_1} - \nabla \bar{u}_{j_2} \right|^p dxdt \nonumber \\
			&\leq c \left( \int_{\ot} \left( F_{\tau_1}\left(\left|\nabla \bar{u}_{j_1} \right|^2 \right) \nabla \bar{u}_{j_1} - F_{\tau_2}\left(\left|\nabla \bar{u}_{j_2} \right|^2\right) \nabla \bar{u}_{j_2} \right) \cdot \left( \nabla \bar{u}_{j_1} - \nabla \bar{u}_{j_2} \right) dxdt \right)^{\frac{p}{2}} \nonumber \\
			&\leq c \left( \int_{\ot} \left(\tau_1 \bar{u}_{j_1} - \tau_2 \bar{u}_{j_2} \right)\left(\bar{u}_{j_1} - \bar{u}_{j_2} \right) dxdt \right)^{\frac{p}{2}} \nonumber \\
			& + c \left( \int_{\ot} \left(\bar{v}_{j_1} - \bar{v}_{j_2} \right)\left( \bar{u}_{j_1} - \bar{u}_{j_2} \right) dxdt \right)^{\frac{p}{2}} \label{nu8}
		\end{align}
		Then, owing to the fact that the Sobolev conjugate of $p$, $p^* =\frac{2p}{2-p} >2$, for all $1<p<2$, we can conclude from lemma \ref{ucomp} and \ref{apriori1} that,
		\begin{equation}
			\buj \rightarrow u, \ \ \mbox{strongly in $L^2\left(\ot\right)$.} \label{nu9}
		\end{equation}
		Then from \eqref{apriori1} we also know that $\tau\buj$ is bounded in $L^2\left(\ot\right)$. This and \eqref{nu9} then imply,
		\begin{align}
			\int_{\ot} \left(\tau_1 \bar{u}_{j_1} - \tau_2 \bar{u}_{j_2} \right) \left( \bar{u}_{j_1} - \bar{u}_{j_2} \right) dxdt & \leq \left\| \tau_1 \bar{u}_{j_1} - \tau_2 \bar{u}_{j_2} \right\|_{2,\ot} \left\| \bar{u}_{j_1} - \bar{u}_{j_2} \right\|_{2,\ot} \nonumber \\
			& \leq c \left\| \bar{u}_{j_1} - \bar{u}_{j_2} \right\|_{2,\ot} \rightarrow 0. \label{nu10}
		\end{align}
		Next since $p>1$, we can derive from lemma \ref{vspace} that $\{\bvj\}$ is bounded in $L^2\left(\ot\right)$. From this and \eqref{nu9} we then have, 
		\begin{equation}
			\int_{\ot} \left(\bar{v}_{j_1} - \bar{v}_{j_2} \right) \left( \bar{u}_{j_1} - \bar{u}_{j_2} \right) dxdt \rightarrow 0. \label{nu11}
		\end{equation}
		Finally putting together \eqref{nu11}, \eqref{nu10}, with \eqref{nu8} finishes the proof of the proposition. 
	\end{proof}

	Summarizing the proceeding propositions, we have upon passing to appropriate sub-sequences, which we won't relabel,
	\begin{align}
		\buj \rightarrow u, \ \ \mbox{ strongly in $C\left(\left[0,T\right];L^p\left(\Omega\right)\right)$ and a.e. on $\ot$,} \label{conv1} \\
		\tuj \rightarrow u, \ \ \mbox{ strongly in $C\left(\left[0,T\right];L^p\left(\Omega\right)\right)$ and a.e. on $\ot$,} \label{conv2} \\
		\nabla \buj \rightarrow \nabla u, \ \ \mbox{ strongly in $L^p\left(\ot\right)$ and a.e. on $\ot$, } \label{conv3} \\
		\bvj \rightarrow v, \ \ \mbox{ weakly in $L^{2}\left(0,T;W^{1,\frac{2p}{p+1}}\left(\Omega\right)\right)$. } \label{conv4} 
	\end{align}
	
	From \eqref{conv3} along with the fact that $M\left(\nabla \buj\right)$ is uniformly bounded we can conclude that, 
	\begin{equation}
		M\left(\nabla \buj \right) \rightarrow M\left(\nabla u \right), \ \ \mbox{ strongly in $\left(L^{\delta}\left(\ot\right)\right)_{2\times 2} $ for all $1\leq \delta < \infty$, and a.e. on $\ot$. } \label{conv5}
	\end{equation}
	Then utilizing \eqref{conv5} and \eqref{conv4} we can conclude that, 
	\begin{equation}
		M\left(\nabla \buj \right) \nabla \bvj \rightarrow M\left(\nabla u \right) \nabla v, \ \ \mbox{ weakly in $L^{\delta}\left(\ot\right)$ for each $1\leq \delta < \frac{2p}{p+1}$. } \label{conv7} 
	\end{equation}
	Then, from \ref{ucomp} we easily have that $\tau \buj \rightarrow 0$ strongly in $L^p\left(\ot\right)$ and a.e on $\ot$. Thus, together with \eqref{conv3} then implies, 
	\begin{equation}
		-\Delta_p \buj \rightarrow v, \ \ \mbox{ weakly in $L^{2}\left(0,T;W^{1,\frac{2p}{p+1}}\left(\Omega\right)\right)$ . }
	\end{equation}
	With this, we are finally ready to pass to the limit in our sequence of approximate solutions. This completes the proof of our main theorem.


\begin{thebibliography}{}
		
		
		
		
		
		\bibitem{BCF} W. K. Burton, N. Cabrera, and F. C. Frank, {\em The growth of crystals and the equilibrium structure of their surfaces}, Philosophical Trans. Royal soc. London A: Mathematical, Physical and Engineering Sciences, {\bf 243} (1951), no. 866, 299-358. 
		
		\bibitem{YG} Yuan Gao, {\em Global strong solution with BV derivatives to singular solid-on-solid model with exponential nonlinearity}, J. Differential Equations, {\bf 267} (2019), 4429-4447. 
		
		\bibitem{GLL} Y. Gao, J.-G. Liu, and X. Y. Lu, {\em Gradient flow approach to an exponential thin film equation: global existence and latent singularity}, ESAIM: Control, Optimisation and Calculus of Variations, {\bf 25} (2019), 49-. arXiv:1710.06995.
		
		\bibitem{GLL2} Y. Gao, J.-G. Liu, and J. Lu, {\em Weak solutions of a continuum model for vicinal surface in the ADL regime}, SIAM J. Math. Anal., {\bf 49} (2017), 1705-1731. 
		
		\bibitem{GT} D. Gilbarg and N.S. Trudinger, {\em Elliptic Partial Differential Equations of Second Order}, Springer-Verlag, Berlin, (1983).
		
		\bibitem{BM} R. Granero-Belinchón and M. Magliocca, {\em Global existence and decay to equilibrium for some crystal surface models}, {\bf 39} (2019), 2101-2131.
		
		\bibitem{KDM} J. Krug, H.T. Dobbs, and S. Majaniemi, {\em Adatom mobility for the solid-on-solid model}, Z. Phys. B {\bf 97} (1995), 281-291.
		
		\bibitem{LS} J.-G. Liu and R. Strain, {\em Global stability for solutions to the exponential PDE describing epitaxial growth}, Interfaces and Free Boundaries, {\bf 21} (2019), 61-68. 
		
		\bibitem{LX} J.-G. Liu, and X. Xu, {\em Existence Theorems For A Multidimensional Crystal Surface Model, } SIAM J. Math. Anal., {\bf 48} (2016), 3667-3687. 
		
		\bibitem{MK} D. Margetis and R. V. Kohn, {\em Continuum relaxation of interacting steps on crystal surfaces in $2+1$ dimensions}, Multiscale Modeling \& Simulation, {\bf 5} (2006), no. 3, 729-758.
		
		\bibitem{MW} J.L. Marzuola and J. Weare, {\em Relaxation of a family of broken-bond crystal surface models},Physical Review, E {\bf 88} (2013), 032403. 
		
		\bibitem{JOden} J. T. Oden, {\em Qualitative Methods in Nonlinear Mechanics}, Prentice-Hall, Inc, New Jersey, 1986. 
		
		\bibitem{PX} B. C. Price and X. Xu, {\em Strong solutions to a fourth order exponential PDE describing epitaxial growth}, Journal of Differential Equations, {\bf 306} (2022) 220-250.
		
		\bibitem{PX2} B. C. Price and X. Xu, {\em Exponential crystal relaxation model with p-laplacian}, Z. Angew. Math. Phys. {\bf 74} 140 (2023) 
		
		\bibitem{S} J. Simon, {\it Compact sets in the space $L^p(0,T;B)$}, Ann. Mat. Pura Appl., {\bf 146}(1987), 65-96.
		
		\bibitem{XX} X. Xu, {\em Partial Regularity for an Exponential PDE in Crystal Surface Models}, Nonlinearity, {\bf 35} 4392, (2022). 
		
		\bibitem{XX2} X. Xu, {\em
			Mathematical validation of a continuum model for relaxation of interacting steps in crystal surfaces in 2 space dimensions},  Calc. Var., {\bf  59}, 158, (2020).
		
		\bibitem{XX3} X. Xu, {\em Existence Theorems for a Crystal Surface Model Involving the $p$
			-Laplace Operator},   SIAM J. Math. Anal., {\bf 50} (2018), 4261–4281.
		
		
		
		
		
		
		
		
		
		
		
		
		
		
		
		
		
		
		
		
		
		
		
	\end{thebibliography}
\end{document}